\documentclass{article}%
\usepackage{amsmath}
\usepackage{amsfonts}
\usepackage{amssymb}
\usepackage{graphicx}%
\providecommand{\U}[1]{\protect\rule{.1in}{.1in}}
\newtheorem{theorem}{Theorem}
\newtheorem{acknowledgement}[theorem]{Acknowledgement}

\newtheorem{claim}[theorem]{Claim}

\newtheorem{corollary}[theorem]{Corollary}

\newtheorem{definition}[theorem]{Definition}

\newtheorem{lemma}[theorem]{Lemma}

\newtheorem{problem}[theorem]{Problem}
\newtheorem{proposition}[theorem]{Proposition}

\newenvironment{proof}[1][Proof]{\noindent\textbf{#1.} }{\ \rule{0.5em}{0.5em}}
\begin{document}

\date{}
\title{No covering with nowhere dense \textsf{P}-sets in the Cohen model}
\author{Alan Dow \thanks{\textit{keywords:} Nowhere dense P-sets, ultrafilters,
towers, Cohen forcing, random forcing. \newline\textit{AMS Classification:}
54A35, 03E35, 03E75, 03E05.}
\and Osvaldo Guzm\'{a}n \thanks{The second author was supported by the PAPIIT grant
IA 104124 and the SECIHTI grant CBF2023-2024-903.}}
\maketitle

\begin{abstract}
We prove that if less than $\aleph_{\omega}$-many Cohen reals are added to a
model of \textsf{CH}, then $\omega^{\ast}$ can not be covered by nowhere dense
\textsf{P}-sets (equivalently, there is an ultrafilter on $\omega$ that does
not contain a tower).

\end{abstract}

\section{Introduction}

Recall that a subset $B$ of a topological space\footnote{All spaces under
discussion are Hausdorff.} $X$ is called a \emph{P-set }if the intersection of
countably many neighborhoods of $B$ is still a neighborhood of it. Similarly,
a point $a\in X$ is a \textsf{P}\emph{-point }if the singleton $\left\{
a\right\}  $ is a \textsf{P}-set. While there are infinite spaces in which
every point is a \textsf{P}-point, this is not the case for compact spaces; in
other words, every infinite compact space contains a non \textsf{P}-point (see
\cite{ConcerningRingsofFunctions}). This result lead Kunen, van Mill and Mills
(see \cite{NowhereDenseClosedPsets}) to wonder the following:\qquad
\qquad\ \ \ \ \qquad\qquad\ \ \ 

\begin{center}
\textit{Can an infinite compact space be covered with}

\textit{nowhere dense \emph{P-}sets?}

\qquad\qquad\ \ \ \ \ \qquad\qquad\qquad
\end{center}

A nowhere dense set is a set whose closure has empty interior, such sets are
often considered \textquotedblleft small\textquotedblright\ in the topological
sense. The motivation behind the question is that, although an infinite
compact space must have non \textsf{P}-points, it might still be possible that
every point lies in a \textquotedblleft small\textquotedblright\ \textsf{P}%
-set. A particular case of the main theorem of \cite{NowhereDenseClosedPsets}
is the following:

\begin{theorem}
[Kunen, van Mill, Mills \cite{NowhereDenseClosedPsets}]If $X$ is a compact
space of $\pi$-weight $\omega_{1},$ then it has a point that does not belong
to any \textit{nowhere dense \emph{P-}set. \label{KunenvanMill}}
\end{theorem}

We now turn our attention to $\omega^{\ast}=\beta\omega\setminus\omega,$ the
remainder of the \v{C}ech-Stone compactification of the countable discrete
space $\omega$ (for an excellent introduction to this topic, see
\cite{betaomega}). Using \emph{Stone duality} (see \cite{LibroBooleanAlgebras}%
), we can translate the assertion \textquotedblleft$\omega^{\ast}$\textit{ can
be covered by nowhere dense \textsf{P}-sets}\textquotedblright\ into a more
combinatorial statement. A \textsf{P}-set in $\omega^{\ast}$ corresponds to a
\textsf{P}-filter\footnote{The undefined notions will be reviewed in the
following section.} and a nowhere dense \textsf{P}-set corresponds to a
nowhere dense \textsf{P}-filter. It follows that the following statements are equivalent:

\begin{enumerate}
\item $\omega^{\ast}$ can not be covered by nowhere dense \textsf{P}-sets.

\item There is an (non-principal) ultrafilter on $\omega$ that does not
contain a nowhere dense \textsf{P}-filter.

\item There is an (non-principal) ultrafilter on $\omega$ that intersects
every tall \textsf{P}-ideal.
\end{enumerate}

\qquad\ \ \qquad\ \ 

By Theorem \ref{KunenvanMill}, it follows that the \emph{Continuum
Hypothesis}\ (\textsf{CH) }implies that $\omega^{\ast}$ can not be covered by
nowhere dense \textsf{P}-sets. Surprisingly, it is also consistent that
$\omega^{\ast}$ can be covered by such sets. This remarkable result was first
proved by Balcar, Frankiewicz and Mills in \cite{MoreonNowhereDensePsets}.
Since then, this topic has continued to attract attention and more results
have been obtained. We now list some of them:

\begin{theorem}
[Balcar, Frankiewicz, Mills \cite{MoreonNowhereDensePsets}]It is consistent
that $\omega^{\ast}$ can be covered by nowhere dense \textsf{P}-sets.
\end{theorem}

\begin{theorem}
[Dow \cite{CozeroAccesiblePoints}]It is consistent with \emph{Martin's Axiom
(MA) }and $\mathfrak{c}=\omega_{2}$ that $\omega^{\ast}$ can not be covered by
nowhere dense \textsf{P}-sets.
\end{theorem}

\begin{theorem}
[Dow \cite{PFiltersCohenRandomLaver}]If $\omega_{2}$ Cohen reals are added to
a model of \textsf{CH + }$\square_{\omega_{1}},$ then $\omega^{\ast}$ can not
be covered by nowhere dense \textsf{P}-sets. \label{Alan Cohen}
\end{theorem}

\begin{theorem}
[Dow \cite{PFiltersCohenRandomLaver}]In the Laver model, every ultrafilter
contains a tower (hence, $\omega^{\ast}$ can be covered with nowhere dense
\textsf{P}-sets).
\end{theorem}

\begin{theorem}
[Zhu \cite{RemarkonNowheredensePsets}, Brendle, Farkas, Verner
\cite{TowersinFilters}]The principle of \emph{Near Coherence of Filters
(}\textsf{NCF}\emph{)} implies that $\omega^{\ast}$ can be covered by nowhere
dense \textsf{P}-sets.
\end{theorem}

\begin{theorem}
[Dow, van Mill \cite{NowhereDenseCCCPSets}]No compact space can be covered by
nowhere dense ccc \textsf{P}-sets.
\end{theorem}

\qquad\ \qquad\ \ 

The purpose of this paper is to improve upon Theorem \ref{Alan Cohen}. We show
that if fewer than $\aleph_{\omega}$ Cohen reals are added to a model of
\textsf{CH, }then there is an ultrafilter that intersects every tall
\textsf{P}-ideal (the principle $\square_{\omega_{1}}$ is no longer needed).
Our approach builds on the argument presented in
\cite{PFiltersCohenRandomLaver}. However, by carrying out a more refined
analysis of the interaction between the countable elementary submodels used in
that proof, we were able to eliminate the need of $\square_{\omega_{1}}$ and
extend the result for more than $\omega_{2}$ Cohen reals. Although our proof
is heavily inspired by \cite{PFiltersCohenRandomLaver}, no familiarity with
that paper is needed to follow the arguments we present here. This work shares
some conceptual overlap with our earlier paper
\cite{UltrafiltersinRandomModel}. While we will reference several results from
that work, the present paper remains self contained, readers only need to
accept the quoted results and require no prior familiarity with
\cite{UltrafiltersinRandomModel}.

\ \ \ \ \qquad\qquad\qquad\ \ \ \qquad\ \ \ \ \ \ \qquad\ \ \ 

The paper is organized as follows. After a review of notation and some
preliminaries, in Section \ref{SeccionCH} we present a proof that \textsf{CH
}implies that $\omega^{\ast}$ can not be covered by nowhere dense
\textsf{P}-sets. This result is not new and follows from Theorem
\ref{KunenvanMill} above. However, we included this result since the proof
serves as motivation for some of the ideas that will be used in the main part
of the paper. In Section \ref{Seccion submodelos}, we develop the main
technical results that we will need in the next section. We obtain several
combinatorial results regarding countable elementary submodels that may have
more applications in the future. In Section
\ref{Seccion ultrafiltros y torres}, we use the ideas and results previously
developed to prove that after adding less $\aleph_{\omega}$ Cohen or random
reals to a model of \textsf{CH}, there is an ultrafilter that intersects every
increasing tower. We use this result in Section \ref{Seccion Covering} to
conclude that if fewer than $\aleph_{\omega}$-many Cohen reals are added to a
model of \textsf{CH}, then $\omega^{\ast}$ can not be covered by nowhere dense
\textsf{P}-sets. We finish the paper with some open questions.

\section{Notation}

If $X$ is any set, we denote by $\mathcal{P}\left(  X\right)  $ its power set
(the collection of all its subsets). We say that $\mathcal{I}\subseteq
\mathcal{P}\left(  X\right)  $ is an \emph{ideal on }$X$ if $\emptyset
\in\mathcal{I}$, $X\notin\mathcal{I}$ and it is closed under finite
intersections and taking subsets. On the other hand, $\mathcal{F\subseteq}$
$\mathcal{P}\left(  X\right)  $ is a \emph{filter on }$X$ if $X\in\mathcal{F}%
$, $\emptyset\notin\mathcal{F},$ and it is closed under finite unions and
taking supersets. An \emph{ultrafilter} is a maximal filter that does not
contain any finite set (in this work, all our ultrafilters are assumed to be
non-principal). Let $\mathcal{B}$ be a collection of subsets of $\omega,$
denote $\mathcal{B}^{\ast}=\left\{  \omega\setminus B\mid B\in\mathcal{B}%
\right\}  .$ It is easy to see that if $\mathcal{F}$ is a filter then
$\mathcal{F}^{\ast}$ is an ideal (which we call the \emph{dual ideal of
}$\mathcal{F}$) and if $\mathcal{I}$ an ideal, then $\mathcal{I}^{\ast}$ is a
filter (\emph{the dual filter of} $\mathcal{I}$). For $A,B\subseteq\omega,$
define $A\subseteq^{\ast}B$ ($A$ \emph{is an almost subset of} $B$) if
$A\setminus B$ is finite. For $\mathcal{P\subseteq}$ $\left[  \omega\right]
^{\omega}$ and $A,B\subseteq\omega,$ we say that $A\ $\emph{is a
pseudointersection of }$\mathcal{P}$ if it is almost contained in all elements
of $\mathcal{P},$ while $B$ is a \emph{pseudounion} of $\mathcal{P}$ if it
almost contains every element of $\mathcal{P}.$ We say $\mathcal{C}%
\subseteq\left[  \omega\right]  ^{\omega}$ is\emph{ centered} if the
intersection of every finite subfamily of $\mathcal{C}$ is infinite. Every
centered family can be extended to an ultrafilter.

\begin{definition}
Let $\mathcal{I}$ be an ideal and $\mathcal{F}$ a filter, both on $\omega.$

\begin{enumerate}
\item $\mathcal{I}$ is a \textsf{P}\emph{-ideal} if every countable subfamily
of $\mathcal{I}$ has a pseudounion in $\mathcal{I}.$

\item $\mathcal{F}$ is a \textsf{P}\emph{-filter} if every countable subfamily
of $\mathcal{F}$ has a pseudointersection in $\mathcal{F}.$

\item $\mathcal{I}$ is \emph{tall }if for every $A\in\left[  \omega\right]
^{\omega},$ there is $B\in\mathcal{I}$ such that $A\cap B$ is infinite.

\item $\mathcal{F}$ is \emph{nowhere dense }if it has no pseudointersection.
\end{enumerate}
\end{definition}

\qquad\ \qquad\ \ \ \ 

It is easy to see that $\mathcal{I}$ is a \textsf{P}-ideal if and only if
$\mathcal{I}^{\ast}$ is a \textsf{P}-filter and $\mathcal{I}$ is tall if and
only if $\mathcal{I}^{\ast}$ is nowhere dense. Important examples of tall
\textsf{P}-ideals (nowhere dense \textsf{P}-filters) are the following:

\begin{definition}
Let $\kappa$ be a cardinal.\qquad\ \ \ \qquad\qquad\ \ \ \ 

\begin{enumerate}
\item $\mathcal{A=}$ $\left\{  A_{\alpha}\mid\alpha<\kappa\right\}
\subseteq\left[  \omega\right]  ^{\omega}$ is a \emph{(decreasing) tower} if
$A_{\beta}\subseteq^{\ast}A_{\alpha}$ whenever $\alpha<\beta$ and
$\mathcal{A}$ has no infinite pseudointersection.

\item $\mathcal{B=}$ $\left\{  B_{\alpha}\mid\alpha<\kappa\right\}
\subseteq\left[  \omega\right]  ^{\omega}$ is an \emph{increasing tower} if
$B_{\alpha}\subseteq^{\ast}B_{\beta}$ whenever $\alpha<\beta$ and for every
$X\in\left[  \omega\right]  ^{\omega},$ there is $\gamma<\kappa$ such that
$B_{\alpha}\cap X$ is infinite (equivalently, any pseudounion of $\mathcal{B}$
is cofinite).
\end{enumerate}
\end{definition}

\qquad\qquad\ \ 

It is easy to see that the ideal generated by an increasing tower is a tall
\textsf{P}-ideal, while the filter generated by a tower is a nowhere dense
\textsf{P}-filter. With a very easy recursive construction, we can prove the following:

\begin{proposition}
[\textsf{CH}]Every tall \textsf{P}-ideal contains an increasing tower, and
every nowhere dense \textsf{P}-filter contains a tower.
\label{CH pideal contiene torre}
\end{proposition}

\section{Review on elementary submodels}

Elementary submodels are now an essential tool for topologists and set
theorists. Their use is so fundamental that many arguments in topology and
infinite combinatorics would be impossible without them. This is the case in
our paper. We now review the basic facts about elementary submodels needed for
our work, and refer the reader to the survey \cite{AlanSubmodelos} to learn
more about this important technique and its many applications.

\qquad\qquad\ \ \ \ \qquad\ \ \ \ \ \ \ \ 

If $\kappa$ is a cardinal, define \textsf{H}$\left(  \kappa\right)  $ as the
set consisting of all sets whose transitive closure has size less than
$\kappa.$ For a detailed treatment of these models, see the book \cite{Kunen}.
For convenience, fix $\trianglelefteq$ a well order of \textsf{H}$\left(
\kappa\right)  $ and denote by \textsf{Sub}$\left(  \kappa\right)  $ the
collection of all countable $M\subseteq$ \textsf{H}$\left(  \kappa\right)  $
such that $\left(  M,\in,\trianglelefteq\right)  $ is an elementary submodel
of $($\textsf{H}$\left(  \kappa\right)  ,\in,\trianglelefteq).$ Let $M\in$
\textsf{Sub}$\left(  \kappa\right)  .$ The \emph{height of }$M$ is defined
$\delta_{M}=M\cap\omega_{1}.$

\begin{definition}
Let $\delta$ be an ordinal. We say that $\delta$ is \emph{indecomposable }if
it is closed under (ordinal) addition. Equivalently, $\gamma+\delta=\delta$
for every $\gamma<\delta.$
\end{definition}

\qquad\qquad\qquad

The height of a model is an indecomposable ordinal (of course it is much more,
since it is also closed under ordinal multiplication, exponentiation and many
more operations). We will need the following result on indecomposable
ordinals, see Chapter 9 of the book \cite{KomjathBook} for a proof.

\begin{proposition}
Let $\delta$ be an indecomposable ordinal. If we split $\delta$ in finitely
many parts, then one part is isomorphic (as a linear order) to $\delta.$
\label{Indescomp}
\end{proposition}

\qquad\qquad\qquad\ \ \ \ 

A key property of countable elementary submodels is that for every $M\in$
\textsf{Sub}$\left(  \kappa\right)  $ and $A\in M,$ if $A$ is countable, then
$A\subseteq M.$ This property will be used implicitly throughout the paper. If
$B$ is a set of ordinals, by $\overline{B}$ we denote its closure (in the
usual topology of the ordinals) and \textsf{OT}$\left(  B\right)  $ denotes
its order type. The following result is very well-known, we include a proof
just for completeness.

\begin{lemma}
If $B$ is a countable set of ordinals, then $\overline{B}$ is also countable.
\end{lemma}

\begin{proof}
Let $\gamma=$ \textsf{OT}$\left(  B\right)  $ (note that it is smaller than
$\omega_{1}$) and $e:\gamma\longrightarrow B$ the (unique) isomorphism. For
$\alpha\leq\gamma,$ denote $\beta_{\alpha}=\bigcup e\left[  \alpha\right]  .$
It follows that $\overline{B}=\left\{  \beta_{\alpha}\mid\alpha\leq
\gamma\right\}  .$
\end{proof}

\qquad\qquad\ \ \ 

In this way, if $M\in$ \textsf{Sub}$\left(  \kappa\right)  $ and $B\in M$ is a
countable set of ordinals, then $\overline{B}$ is contained in $M.$ The
following result is folklore (a proof can be found in
\cite{UltrafiltersinRandomModel}).

\begin{lemma}
[\textsf{CH}]Let $M,N$ be two elements of \textsf{Sub}$\left(  \kappa\right)
.$ If $\delta_{M}\leq\delta_{N},$ then \textsf{H}$\left(  \omega_{1}\right)
\cap M\subseteq$ \textsf{H}$\left(  \omega_{1}\right)  \cap N.$
\label{Lema contencion Hw1}
\end{lemma}

\qquad\qquad\qquad\ \ \qquad\ \ 

A set $C\subseteq\omega_{1}$ is called a \emph{club }if it is closed and
unbounded. We will say that $\left\{  M_{\alpha}\mid\alpha\in\omega
_{1}\right\}  \subseteq$ \textsf{Sub}$\left(  \kappa\right)  $ is a
\emph{continuous increasing chain }if the following conditions hold:

\begin{enumerate}
\item $M_{\alpha}\in M_{\beta}$ whenever $\alpha<\beta.$

\item $M_{\alpha}=\bigcup\limits_{\xi<\alpha}M_{\xi}$ in case $\alpha$ is limit.
\end{enumerate}

\qquad\ \qquad\ \ \ \ \qquad\qquad\qquad\ \ \qquad\qquad\ \ 

It is easy to see that if $\left\{  M_{\alpha}\mid\alpha\in\omega_{1}\right\}
\subseteq$ \textsf{Sub}$\left(  \kappa\right)  $ is a continuous increasing
chain, then $\left\{  \delta_{M_{\alpha}}\mid\alpha<\omega_{1}\right\}  $ is a
club. Moreover, every club on $\omega_{1}$ contains one of this form.

\section{Review on Cohen and random forcing}

We adopt Kunen's presentation of Cohen and random forcing from
\cite{KunenCohenyRandom}, which enables us to treat both forcing notions
uniformly. Proofs of the results mentioned in this section can be found in
that paper. For any set $I$, we endow the product space $2^{I}$ with its
standard product topology. Recall that $B\subseteq$ $2^{I}$ is \emph{Baire} if
it belongs to the smallest $\sigma$-algebra that contains all clopen sets.
\textsf{Baire}$\left(  2^{I}\right)  $ is the collection of all the Baire
subsets of $2^{I}.$ By $\mathcal{M}_{I}$ we denote the $\sigma$-ideal of
meager sets in $2^{I}$ and $\mathcal{N}_{I}$ is the $\sigma$-ideal of null
sets in $2^{I}$ (where $2^{I}$ has the standard product measure). It is not
hard to see that both ideals have a Baire base ($F_{\sigma}$ for
$\mathcal{M}_{I}$ and $G_{\delta}$ for $\mathcal{N}_{I}$). By $\mathbb{C}%
\left(  I\right)  $ (\emph{Cohen forcing on} $I$) we denote the quotient
\textsf{Baire}$\left(  2^{I}\right)  \diagup\mathcal{M}_{I},$ while
$\mathbb{B}\left(  I\right)  $ (\emph{random forcing on} $I$) is
\textsf{Baire}$\left(  2^{I}\right)  \diagup\mathcal{N}_{I}.$ We order
$\mathbb{C}\left(  I\right)  $ ($\mathbb{B}\left(  I\right)  $) by inclusion:
$A\leq B$ if and only if $A\subseteq B$. Note that as defined, $\mathbb{C}%
\left(  I\right)  $ and $\mathbb{B}\left(  I\right)  $ are not separative
partial orders. We could also use the quotients instead. This choice does not
significantly affect the content of the paper.

\qquad\ \ \ \ \qquad\ \ \ \ \ \ \ \ 

For $I,J$ two infinite sets and $\triangle:I\longrightarrow J$ an injective
function, define $\triangle^{\prime}:2^{J}\longrightarrow2^{I}$ by
$\triangle^{\prime}\left(  f\right)  =f\circ\triangle$ (for $f\in2^{J}).$
\ Furthermore, we also define $\triangle_{\ast}:\mathcal{P}\left(
2^{I}\right)  \longrightarrow\mathcal{P}\left(  2^{J}\right)  $ as
$\triangle_{\ast}\left(  A\right)  =\left\{  f\in2^{J}\mid f\circ\triangle\in
A\right\}  .$ It is easy to see that $\triangle_{\ast}\left(  A\right)
=\left(  \triangle^{\prime}\right)  ^{-1}\left(  A\right)  $.

\begin{lemma}
Let $I,J$ two infinite sets and $\triangle:I\longrightarrow J$ a bijective
function. The function $\triangle^{\prime}$ is a homeomorphism.
\end{lemma}

\qquad\ \qquad\ \ \ 

We also have the following:

\begin{lemma}
Let $I,J$ and $K$ be infinite sets and $\mathbb{P}$ be either $\mathbb{C}$ or
$\mathbb{B}.$

\begin{enumerate}
\item If $\triangle:I\longrightarrow J$ is bijective, then $\triangle_{\ast
}:\mathbb{P}\left(  I\right)  \longrightarrow\mathbb{P}\left(  J\right)  .$

\item If $K\subseteq I,$ then $\mathbb{P}\left(  K\right)  $ is isomorphic to
a regular suborder of $\mathbb{P}\left(  I\right)  $ (for convenience, we view
them as real suborders).

\item Let $\dot{a}$ be a $\mathbb{P}\left(  I\right)  $ name for a subset of
$\omega.$ There is \ $L\in\left[  I\right]  ^{\omega}$ such that $\dot{a}$ is
a $\mathbb{P}\left(  L\right)  $ name.
\end{enumerate}
\end{lemma}

\qquad\ \ \ \ \ 

If $\dot{a}$ and $L$ are as in the third point, we say that $L$\emph{ is a
support of }$\dot{a}.$ Let $\mathbb{P}$ be a partial order. By $V^{\mathbb{P}%
}$ we denote the class of all $\mathbb{P}$-names (see \cite{oldKunen}). An
automorphism $F:\mathbb{P\longrightarrow P}$ can be extended to a permutation
of $V^{\mathbb{P}},$ which we also denote by $F$ (see \cite{Jech} or
\cite{JechAC} for details). An important fact of isomorphisms is the following:

\begin{proposition}
Let $\mathbb{P}$ be a partial order and $F:\mathbb{P\longrightarrow P}$ an
automorphism. For every $p\in\mathbb{P},$ $\varphi$ a formula, $\dot{a}%
_{0},...,\dot{a}_{n}\in V^{\mathbb{P}}$ and sets $b_{0},...,b_{m},$ the
following are equivalent: \label{Prop formula isomorfismo}

\begin{enumerate}
\item $p\Vdash$\textquotedblleft$\varphi(\dot{a}_{0},...,\dot{a}_{n}%
,b_{0},...,b_{m})$\textquotedblright.

\item $F\left(  p\right)  \Vdash$\textquotedblleft$\varphi(F\left(  \dot
{a}_{0}\right)  ,...,F\left(  \dot{a}_{n}\right)  ,b_{0},...,b_{m}%
)$\textquotedblright.\ 
\end{enumerate}
\end{proposition}

\qquad\qquad\ \ \ 

Regarding automorphisms of Cohen or random forcing, we will need the following:

\begin{lemma}
Let $I$ be an infinite set, $\triangle,\sigma:I\longrightarrow I$ bijections,
$\dot{a}$ a $\mathbb{C}\left(  I\right)  $ name for a subset of $\omega$ and
$J\in\left[  I\right]  ^{\omega}$ a support of $\dot{a}.$ If $\triangle
\upharpoonright J=\sigma\upharpoonright J,$ then $\triangle_{\ast}(\dot
{a})=\sigma_{\ast}(\dot{a}).$ \label{automorfismo coinciden}
\end{lemma}

\qquad\ \ \ \ \ \ 

The following is an important result attributed to Kunen. For a proof in the
case of Cohen forcing, see \cite{Kunen} and for the case of random forcing,
see \cite{MeasuresinducebbyForcingnames}.

\begin{theorem}
[Kunen]Assume \textsf{CH. }Let $\kappa$ be a regular uncountable cardinal.
$\mathbb{C}\left(  \kappa\right)  $ ($\mathbb{B}\left(  \kappa\right)  $)
forces that every tower has length $\omega_{1}.$ \label{KunenTorres}
\end{theorem}

\qquad\qquad\ \ \ 

The theorem is particularly interesting for random forcing, since while Cohen
forcing preserves towers\footnote{We say that a forcing $\mathbb{P}$
\emph{preserves towers }if every tower from the ground model is still a tower
after forcing with $\mathbb{P}.$}, random forcing does not.

\section{Covering with nowhere dense P-sets under CH\textsf{\label{SeccionCH}%
}}

In this section we prove that \textsf{CH }implies that $\omega^{\ast}$ can not
be covered by nowhere dense \textsf{P}-sets. As mentioned before, this result
is not new and follows directly from Theorem \ref{KunenvanMill}. Moreover, our
proof is not really new, since the argument easily follows from the proof of
Theorem 4.1 of \cite{PFiltersCohenRandomLaver}. Nevertheless, we include it
here because it motivates some of the ideas we will develop in subsequent sections.

\begin{theorem}
[Kunen, van Mill, Mills]\textsf{CH }implies that $\omega^{\ast}$ can not be
covered by nowhere dense \textsf{P}-sets.
\end{theorem}

\begin{proof}
By Proposition \ref{CH pideal contiene torre}, it is enough to construct an
ultrafilter that intersects every increasing tower. Let $\mathbb{T}$ be the
family of all increasing towers. For $\mathcal{A}\in\mathbb{T},$ we enumerate
it as $\mathcal{A}=\{B^{\mathcal{A}}\left(  \alpha\right)  \mid\alpha\in
\omega_{1}\}.$ Furthermore, we choose $N\left(  \mathcal{A}\right)  \in$
\textsf{Sub}$\left(  \omega_{2}\right)  $ such that $\mathcal{A}\in N\left(
\mathcal{A}\right)  .$ Denote $\delta_{\mathcal{A}}=\delta_{N\left(
\mathcal{A}\right)  }$ and $B^{\mathcal{A}}=B^{\mathcal{A}}(\delta
_{\mathcal{A}}).$

\qquad\qquad\qquad\qquad

Let $\mathcal{B=\{}B^{\mathcal{A}}\mid\mathcal{A\in\mathbb{T}\}}.$ We claim
that $\mathcal{B}$ is a centered family. Once this is proved, any ultrafilter
extending $\mathcal{B}$ will be as desired.

\begin{claim}
Let $\mathcal{A}_{0},...,\mathcal{A}_{n}\in\mathbb{T}.$ Denote $\delta
_{i}=\delta_{\mathcal{A}_{i}}$ and assume $\delta_{i}\leq\delta_{i+1}$ for
$i<n.$ There are $\beta_{i}<\delta_{i}$ such that $B_{\beta_{0}}%
^{\mathcal{A}_{0}}\cap...\cap B_{\beta_{n}}^{\mathcal{A}_{n}}$ is infinite.
\end{claim}

\qquad\ \qquad\ \ \ 

We prove the claim by induction. For $n=0,$ we can take any $\beta_{0}%
<\delta_{0}.$ Assume the claim is true for $n,$ we will prove it is also true
for $n+1.$ Fix $\beta_{i}<\delta_{i}$ for $i\leq n$ such that $C=B_{\beta_{0}%
}^{\mathcal{A}_{0}}\cap...\cap B_{\beta_{n}}^{\mathcal{A}_{n}}$ is infinite.
Since $\beta_{i}<\delta_{i},$ we know that $B_{\beta_{i}}^{\mathcal{A}_{i}}\in
N(\mathcal{A}_{i}).$ Furthermore, since $\delta_{i}\leq\delta_{n+1},$ we can
apply Lemma \ref{Lema contencion Hw1} and conclude that each $B_{\beta_{i}%
}^{\mathcal{A}_{i}}$ (for $i\leq n$) is in $\mathcal{N(A}_{n+1}\mathcal{)},$
so $C$ is in there as well. Finally, since $\mathcal{A}_{n+1}\in
\mathcal{N(A}_{n+1}\mathcal{)}$ and is an increasing tower, by elementarity we
can find $\beta_{n+1}<\delta_{n+1}$ such that $C\cap B_{\beta_{n+1}%
}^{\mathcal{A}_{n+1}}$ is infinite. This finishes the proof of the claim.

\qquad\qquad\ \ \ \qquad\ \ 

We can now prove that $\mathcal{B}$ is centered. Let $\mathcal{A}%
_{0},...,\mathcal{A}_{n}\in\mathbb{T}.$ By the claim above, we find $\beta
_{i}<\delta_{i}$ such that $B_{\beta_{0}}^{\mathcal{A}_{0}}\cap...B_{\beta
_{n}}^{\mathcal{A}_{n}}$ is infinite. Finally, since $B^{\mathcal{A}_{i}}$
almost contains $B_{\beta_{i}}^{\mathcal{A}_{i}},$ we conclude that
$\bigcap\limits_{i\leq n}B^{\mathcal{A}_{i}}$ is infinite.
\end{proof}

\qquad\qquad\ \ \qquad\ \ \ \ 

To summarize the previous proof, we assign an elementary submodel to each
increasing tower. We then show that the family consisting of the tower
elements indexed by the corresponding model's height is centered. We aim to
apply this same approach in a forcing extension. However, we now face the
challenge that the Continuum Hypothesis (which was essential for invoking
Lemma \ref{Lema contencion Hw1}) is no longer available. To resolve this
issue, we will work from the ground model using names for increasing towers
rather than the towers themselves. Naturally, this introduces several
additional complications. A significant portion of the work is devoted to
developing an appropriate method for \textquotedblleft
transferring\textquotedblright\ names between models. Another crucial aspect
is that we will need to impose stronger requirements on the elementary
submodels we work with, as we will see in the next section.

\section{On countable elementary submodels \label{Seccion submodelos}}

For this section, we fix a sufficiently large regular cardinal $\kappa$. Our
goal is to obtain several results concerning the combinatorics and interaction
of countable elementary submodels of \textsf{H}$\left(  \kappa\right)  .$ We
believe that the techniques and results developed here will have applications
beyond the scope of this paper.

\begin{definition}
Let $N$ be a countable subset of \textsf{H}$\left(  \kappa\right)  $ and
$n\in\omega.$ Define, by recursion on $n,$ what it means for $N$ \emph{to be
an }$n$\emph{-model} as follows:

\begin{enumerate}
\item $N$ is a $0$\emph{-model} if $N\in$ \textsf{Sub}$\left(  \kappa\right)
.$

\item $N$ is an $\left(  n+1\right)  $\emph{-model} if $N\in$ \textsf{Sub}%
$\left(  \kappa\right)  $ and there is a continuous and increasing chain
$\left\{  N_{\alpha}\mid\alpha<\delta_{N}\right\}  \subseteq N$ such that $N=%
{\textstyle\bigcup\limits_{\alpha<\delta_{N}}}
N_{\alpha}$ and each $N_{\alpha}$ is an $n$-model. In this situation, we say
that $\left\{  N_{\alpha}\mid\alpha<\delta_{N}\right\}  $ witnesses that $N$
\emph{is an} $\left(  n+1\right)  $\emph{-model.}
\end{enumerate}
\end{definition}

\qquad\qquad\ \ 

Clearly, an $\left(  n+1\right)  $-model is an $n$-model. It is also easy to
see that if $N$ is an $\left(  n+1\right)  $-model and $X$ is any element of
$N,$ then we can find a sequence $\left\{  N_{\alpha}\mid\alpha<\delta
_{N}\right\}  $ witnessing that $N$ is an $\left(  n+1\right)  $-model for
which $X\in N_{0}$ (and hence $X\in N_{\alpha}$ for every $\alpha<\delta_{N}%
$). We reiterate that every $n$-model is countable. It is not hard to find
$n$-models, as the next lemma shows:

\begin{lemma}
Let $X\in$ \textsf{H}$\left(  \kappa\right)  $ and $n\in\omega.$ There is $N$
an $n$-model such that $X\in N.$
\end{lemma}

\begin{proof}
Let $\left\{  M_{\alpha}\mid\alpha\in\omega_{1}\right\}  \subseteq$
\textsf{Sub}$\left(  \kappa\right)  $ be a continuous increasing chain such
that $X\in M_{0}$. Define $C=\left\{  \delta_{\alpha}\mid\alpha\in\omega
_{1}\right\}  ,$ where $\delta_{\alpha}=\delta_{M_{\alpha}}$ for every
$\alpha\in\omega_{1}.$ We now recursively define:

\qquad\qquad\ \ 

\hfill%
\begin{tabular}
[c]{lll}%
$C_{0}$ & $=$ & $C$\\
$C_{n+1}$ & $=$ & $\{\delta\in C_{n}\mid$\textsf{OT}$\left(  \delta\cap
C_{n}\right)  =\delta\}$%
\end{tabular}
\qquad\ \qquad\qquad\hfill

\qquad\qquad\ \ \ \ \qquad\ \ \ \medskip

Note that each $C_{n}$ is a club and $C_{n+1}\subseteq C_{n}$. Moreover, it is
not hard to see that if $\delta_{\alpha}\in C_{n},$ then $M_{\alpha}$ is an
$n$-model.
\end{proof}

\qquad\ \qquad\ \ 

We need the following notion:

\begin{definition}
Let $M_{0},...,M_{n}$ $\in$ \textsf{Sub}$\left(  \kappa\right)  .$ We say that
the sequence $\left\langle M_{0},...,M_{n}\right\rangle $ is $\delta
$\emph{-increasing }if $\delta_{M_{i}}\leq\delta_{M_{i+1}}$ for each $i<n.$
\end{definition}

\qquad\qquad\ \ \ 

We remark that it will often be the case that $\delta_{M_{i}}=\delta_{M_{i+1}%
}$ for some $i<n.$ When working with a $\delta$-increasing
sequence$\ \left\langle M_{0},...,M_{n}\right\rangle ,$ we implicitly assume
that $n\in\omega$ and we will write $\delta_{i}$ instead of $\delta_{M_{i}}$
if there is no risk of confusion. In this paper, a\emph{ partition} $P$ is
simply a collection of pairwise disjoint sets ($\emptyset\in P$ is allowed)
and a partition of a set $A$ is a partition whose union is $A.$

\begin{proposition}
[\textsf{CH}]Assume $l,n\in\omega$ such that $l>0,$ $\left\langle
M_{0},...,M_{n}\right\rangle $ is a $\delta$-increasing sequence and $A_{i}\in
M_{i}\cap\left[  \omega_{l}\right]  ^{\leq\omega}$ for every $i\leq n.$ Let
$N$ be an $l$-model such that $\delta_{n}<\delta_{N}$ and $X\in N$ be any set.
There are $L\in N\ $and $\left\{  A_{i}\left(  u\right)  \mid i\leq n\wedge
u\in2\right\}  $ such that for every $i\leq n,$ the following properties
hold:\label{Partir sucesion}

\begin{enumerate}
\item $L\in$ \textsf{Sub}$\left(  \kappa\right)  ,$ $\delta_{n}<\delta_{L}$
and $X\in L.$

\item $\left\{  A_{i}\left(  0\right)  ,A_{i}\left(  1\right)  \right\}  \in
M_{i}$ and is a partition of $A_{i}.$

\item $A_{i}\left(  0\right)  \in L$ and $A_{i}\left(  1\right)  \cap
L=\emptyset.$
\end{enumerate}
\end{proposition}

\begin{proof}
Before starting the proof, note that (thanks to the Pairing Axiom), it is
equivalent to consider a single set $X$ or finitely many of them, since any
finite collection can be coded as a single set. The proof proceeds by
induction on $l.$ We start with $l=1.$ In here, we have that each $A_{i}$
belongs to \textsf{H}$\left(  \omega_{1}\right)  .$ Since $\delta_{i}%
<\delta_{N}$ for every $i\leq n,$ we know that $A_{0},...,A_{n}\in N$ (see
Lemma \ref{Lema contencion Hw1}). Define $A_{i}\left(  0\right)  =A_{i}$ and
$A_{i}\left(  1\right)  =\emptyset$. Since $N$ is the union of an increasing
chain of submodels, we can find $L\in N$ for which $\delta_{n},A_{0}%
,...,A_{n},X\in L.$ This finishes the proof for the base case of the induction.

\qquad\qquad\qquad\ \ \ \ 

We now assume the proposition is true for $l,$ we will prove it is true for
$l+1$ as well. Given any $\xi\in\omega_{l+1},$ denote by $g_{\xi}%
:\xi\longrightarrow\omega_{l}$ the $\trianglelefteq$-least injective function.
It follows that if $\xi\in M\in$ \textsf{Sub}$\left(  \kappa\right)
,$\ then\ $g_{\xi}\in M$. Since $N$ is an $\left(  l+1\right)  $-model, we can
find a continuous increasing chain $\left\{  N_{\alpha}\mid\alpha<\delta
_{N}\right\}  \subseteq N$ that witnesses that $N$ is an $\left(  l+1\right)
$-model and $X,\delta_{n}\in N_{0}.$

\begin{claim}
There is $\eta<\delta_{N}$ such that for every $i\leq n,$ the following
holds:\medskip

\hfill%
\begin{tabular}
[c]{l}%
$%
{\displaystyle\bigcup}
\left(  N_{\eta}\cap A_{i}\right)  =%
{\displaystyle\bigcup}
\left(  N_{\eta+1}\cap A_{i}\right)  .$%
\end{tabular}
\qquad\qquad\ \qquad\qquad\qquad\hfill
\end{claim}

\qquad\ \ \ \qquad\ \ \ 

We proceed by contradiction. Define $c:\delta_{N}\longrightarrow n+1$ such
that $c\left(  \eta\right)  $ is the least $i$ for which $%
{\displaystyle\bigcup}
\left(  N_{\eta}\cap A_{i}\right)  <%
{\displaystyle\bigcup}
\left(  N_{\eta+1}\cap A_{i}\right)  .$ By Proposition \ref{Indescomp}, we can
find $i\leq n$ for which the set $F=\left\{  \eta\in\delta_{N}\mid c\left(
\eta\right)  =i\right\}  $ is order isomorphic to $\delta_{N}.$ For any
$\eta\in F,$ let $\varepsilon_{\eta}=%
{\displaystyle\bigcup}
\left(  N_{\eta}\cap A_{i}\right)  $ and define $G=\left\{  \varepsilon_{\eta
}\mid\eta\in F\right\}  .$ Note that if $\eta,\xi\in F$ and $\eta<\xi,$ then
the following holds:\medskip

\hfill%
\begin{tabular}
[c]{lll}%
$\varepsilon_{\eta}$ & $=$ & $%
{\displaystyle\bigcup}
\left(  N_{\eta}\cap A_{i}\right)  $\\
& $<$ & $%
{\displaystyle\bigcup}
\left(  N_{\eta+1}\cap A_{i}\right)  $\\
& $\leq$ & $%
{\displaystyle\bigcup}
\left(  N_{\xi}\cap A_{i}\right)  $\\
& $=$ & $\varepsilon_{\xi}$%
\end{tabular}
\qquad\qquad\ \qquad\qquad\qquad\hfill

\qquad\ \ \ \ \qquad\ \ \ \ \ 

In this way, $F$ and $G$ are isomorphic, so the order type of $G$ is
$\delta_{N}.$ Clearly $G\subseteq\overline{A_{i}}$, but this is a
contradiction, since \textsf{OT(}$\overline{A_{i}})<\delta_{i}<\delta_{N}$.
This finishes the proof of the claim.

\qquad\qquad\ \ 

Fix $\eta<\delta_{N}$ as in the claim. For each $i\leq n,$ define:\medskip

\hfill%
\begin{tabular}
[c]{lll}%
$\beta_{i}$ & $=$ & $%
{\displaystyle\bigcup}
\left(  N_{\eta}\cap A_{i}\right)  +1$\\
& $=$ & $%
{\displaystyle\bigcup}
\left(  N_{\eta+1}\cap A_{i}\right)  +1.$%
\end{tabular}
\qquad\qquad\qquad\ \qquad\qquad\hfill

\qquad\ \ \ \ \qquad\ \ \ \ \ 

We now argue that $\beta_{i}\in N_{\eta+1}\cap M_{i}:$ On one hand, $\beta
_{i}\in\overline{A_{i}}\subseteq M_{i},$ and on the other hand, we have that
$\beta_{i}\in\overline{N_{\eta}\cap\omega_{l+1}}\subseteq N_{\eta+1}.$ Define
$B_{i}=g_{\beta_{i}}\left[  A_{i}\cap\beta_{i}\right]  $, which belongs to
$M_{i}$ (since both $\beta_{i}$ and $A_{i}$ are in $M_{i}$). In this way, we
have the following:

\begin{enumerate}
\item $\left\langle M_{0},...,M_{n}\right\rangle $ is a $\delta$-increasing
sequence and $B_{i}\in M_{i}\cap\left[  \omega_{l}\right]  ^{\leq\omega}$.

\item $\delta_{n}<\delta_{N_{\eta+1}}.$

\item $N_{\eta_{+1}}$ is an $l$-model.

\item $X,$ $\beta_{0},...,\beta_{n},N_{\eta}\in N_{\eta_{+1}}.$
\end{enumerate}

\qquad\ \ \ 

We can now invoke the induction hypothesis and find $L\in N_{\eta_{+1}}$ and a
sequence $\left\{  B_{i}\left(  u\right)  \mid i\leq n\wedge u\in2\right\}  $
with the following properties:

\begin{enumerate}
\item $L\in$ \textsf{Sub}$\left(  \kappa\right)  ,$ $\delta_{n}<\delta_{L}.$

\item $\left\{  B_{i}\left(  0\right)  ,B_{i}\left(  1\right)  \right\}  \in
M_{i}$ and is a partition of $B_{i}.$

\item $B_{i}\left(  0\right)  \in L$ and $B_{i}\left(  1\right)  \cap
L=\emptyset.$

\item $X,$ $\beta_{0},...,\beta_{n},N_{\eta}\in L.$
\end{enumerate}

\qquad\ \ \qquad\ \ \ 

Furthermore, since $N_{\eta}\subseteq L\subseteq N_{\eta_{+1}},$ it follows
that $\beta_{i}=%
{\displaystyle\bigcup}
\left(  L\cap A_{i}\right)  +1.$ We now define $A_{i}\left(  0\right)
=g_{\beta_{i}}^{-1}\left(  B_{i}\left(  0\right)  \right)  $ and $A_{i}\left(
1\right)  =A_{i}\setminus A_{i}\left(  0\right)  .$ We will now prove this are
the items we are looking for. Since $\beta_{i}$ and $B_{i}\left(  0\right)  $
are in $M_{i}\cap L,$ it follows that $A_{i}\left(  0\right)  \in M_{i}\cap
L.$ Moreover, since $A_{i}\in M_{i},\ $we conclude that$\ A_{i}\left(
1\right)  $ is also in $M_{i}.$ It remains to prove that $A_{i}\left(
1\right)  \cap L=\emptyset.$ Assume there is $\alpha\in A_{i}\left(  1\right)
\cap L.$ It follows that $\alpha<\beta_{i},$ so we apply it the function
$g_{\beta_{i}}.$ Since $\alpha\notin A_{i}\left(  0\right)  =g_{\beta_{i}%
}^{-1}\left(  B_{i}\left(  0\right)  \right)  ,$ it follows that $g_{\beta
_{1}}\left(  \alpha\right)  \notin B_{i}\left(  0\right)  ,$ which implies
that $g_{\beta_{1}}\left(  \alpha\right)  \in B_{i}\left(  1\right)  .$
Moreover, $\beta_{i},\alpha\in L,$ so $g_{\beta_{i}}\left(  \alpha\right)  \in
L\cap B_{i}\left(  1\right)  ,$ which is a contradiction, since these two sets
are disjoint.
\end{proof}

\qquad\qquad\ \ \ \ \ 

We will need the following notion:

\begin{definition}
Let $l,n\in\omega$ and $\left\langle M_{0},...,M_{n}\right\rangle $ a $\delta
$-increasing sequence. We say that $P=\left\langle P_{i}\mid i\leq
n\right\rangle $ \emph{is a decisive partition (of countable subsets of
}$\omega_{l}$\emph{) for} $\left\langle M_{0},...,M_{n}\right\rangle $ if for
every $i,j\leq n,$ the following conditions hold:

\begin{enumerate}
\item $P_{i}\in M_{i}$, is a finite partition and $P_{i}\subseteq\left[
\omega_{1}\right]  ^{\leq\omega}.$

\item If $i<j$ and $A\in P_{i},$ either $A\in P_{j}$ or $A\cap M_{j}%
=\emptyset.$
\end{enumerate}
\end{definition}

\qquad\ \qquad\ \ \ 

We will omit the phrase \textquotedblleft of countable subsets of $\omega_{l}%
$\textquotedblright\ when $l$ is clear from the context. The following lemma
was implicitly proved in \cite{PFiltersCohenRandomLaver}, for the convenience
of the reader, we provide the argument here.

\begin{lemma}
Let $l,n\in\omega$ and $\left\langle M_{0},...,M_{n}\right\rangle $ a $\delta
$-increasing sequence. If $P=\left\langle P_{i}\mid i\leq n\right\rangle $ is
a decisive partition for $\left\langle M_{0},...,M_{n}\right\rangle ,$ then
$\bigcup\limits_{i\leq n}P_{i}$ is a partition.
\end{lemma}

\begin{proof}
Let $A\in P_{i}$ and $B\in P_{j}$ with $A\neq B$ and $i\leq j.$ We need to
show that $A$ and $B$ are disjoint. If $A\in P_{j},$ then the conclusion is
trivial, so assume that $A\notin P_{j}$. Since $P$ is a decisive partition, it
follows that $A\cap M_{j}=\emptyset.$ Consequently, $A$ and $B$ are disjoint
because $B\subseteq M_{j}.$
\end{proof}

\qquad\ \ \qquad\ \ 

We need a few additional definitions.

\begin{definition}
Let $P$ and $R$ be two partitions. We say $R$ \emph{refines} $P$ (which we
denote by $R\preceq P$) in case the following conditions hold:

\begin{enumerate}
\item Every element of $R$ is contained in one of $P.$

\item Every element of $P$ is the union of some of the elements of $R.$
\end{enumerate}
\end{definition}

\qquad\ \ \ \qquad\ \ \ 

Note that if $R$ refines $P,$ then $\bigcup R=\bigcup P.$ We will say that $R$
$2$\emph{-refines} $P$ (denote by $R\preceq_{2}P)$ if $R\preceq P$ and every
element of $P$ is the union of at most two elements of $R.$ The following is
Lemma 41 of \cite{UltrafiltersinRandomModel}.

\begin{lemma}
[\textsf{CH}]Let $M,N\in$ \textsf{Sub}$\left(  \kappa\right)  $ with
$\delta_{M}\leq\delta_{N}.$ If $A\in M\cap N$ and is a countable subset of the
ordinals, then $\mathcal{P}\left(  A\right)  \cap M\subseteq N.$
\label{Lema potencia contenida}
\end{lemma}

\qquad\ \ \ \ \ 

With the aid of Proposition \ref{Partir sucesion}, we can prove the following:

\begin{corollary}
[\textsf{CH}]Assume $l,n\in\omega$ with $l>0,$ $\left\langle M_{0}%
,...,M_{n}\right\rangle $ is a $\delta$-increasing sequence and
$P=\left\langle P_{i}\mid i\leq n\right\rangle $ is a decisive partition of
countable subsets of $\omega_{l}$ for $\left\langle M_{0},...,M_{n}%
\right\rangle .$ Let $N$ be an $l$-model such that $\delta_{n}<\delta_{N}$ and
$X\in N$ be any set. There are $L\in N\ $and $R=\left\langle R_{i}\mid i\leq
n\right\rangle $ such that the following conditions hold:
\label{Partir Particion}

\begin{enumerate}
\item $L\in$ \textsf{Sub}$\left(  \kappa\right)  ,$ $\delta_{n}<\delta_{L}$
and $X\in L.$

\item $R=\left\langle R_{i}\mid i\leq n\right\rangle $ is a decisive partition
for $\left\langle M_{0},...,M_{n}\right\rangle .$

\item $R_{i}\preceq_{2}P_{i}$ for every $i\leq n.$

\item If $A\in\bigcup\limits_{i\leq n}R_{i},$ then either $A\in L$ or $A\cap
L=\emptyset.$
\end{enumerate}
\end{corollary}

\begin{proof}
Let $k_{i}=\left\vert P_{i}\right\vert $ and $P_{i}=\left\{  B_{i}\left(
v\right)  \mid v<k_{i}\right\}  .$ Consider the sequence $\left\langle
\overline{M}_{0},...,\overline{M}_{m}\right\rangle $ which is obtained from
$\left\langle M_{0},...,M_{n}\right\rangle $ by repeating $M_{i},$ $k_{i}$
many times. We then distribute $P_{i}$ among the copies of $M_{i}.$ Formally, define:

\begin{enumerate}
\item The intervals on the natural numbers $F_{0}=\left[  0,k_{0}-1\right]  $
and\newline$F_{i+1}=\left[  \max\left(  F_{i}\right)  +1,\max\left(
F_{i}\right)  +k_{i+1}-1\right]  $ for every $i\leq n.$

\item $m=\left\vert F_{0}\cup...\cup F_{n}\right\vert .$

\item $\overline{M}_{j}=M_{i}$ if $j\in F_{i}.$

\item $A_{i}=B_{i}\left(  v\right)  $ if $j\in F_{i}$ and $j$ is the $v$-th
element of $F_{i}.$
\end{enumerate}

\qquad\ \qquad\ \ 

Clearly $\left\langle \overline{M}_{0},...,\overline{M}_{m}\right\rangle $ is
$\delta$-increasing and $\delta_{\overline{M}_{m}}<\delta_{N}.$ We can apply
Proposition \ref{Partir sucesion} and find $L\in N$ and $Q=\langle Q_{j}\mid
j\leq m\rangle$ such that for every $j\leq m,$ the following conditions hold:

\begin{enumerate}
\item $L\in$ \textsf{Sub}$\left(  \kappa\right)  ,$ $\delta_{\overline{M}_{m}%
}<\delta_{L}$ and $X\in L.$

\item $Q_{j}\preceq_{2}\left\{  A_{j}\right\}  $ and it is in $\overline
{M}_{j}.$

\item One element of $Q_{j}$ is in $L$ and the other is disjoint from it.
\end{enumerate}

\qquad\ \qquad\ \ \ 

Furthermore, we make sure that if $A_{j}=A_{k},$ then $Q_{j}=Q_{k}.$ This is
possible by Lemma \ref{Lema potencia contenida}. Define $R_{i}=\bigcup
\limits_{j\in F_{i}}Q_{j}$ (for $i\leq n$) and $R=\left\langle R_{0}%
,...,R_{n}\right\rangle .$ It follows that $R_{i}\in M_{i}$ and $R_{i}%
\preceq_{2}P_{i}.$ It remains to prove that $R$ is decisive for $\left\langle
M_{0},...,M_{n}\right\rangle .$ Let $i<j$ and $D\in R_{i}.$ Pick $B\in P_{i}$
such that $D\subseteq B.$ If $B\notin P_{j},$ then $B\cap M_{j}=\emptyset$ and
the same holds for $D.$ Finally, if $B\in P_{j},$ then $D\in R_{j}$ by definition.
\end{proof}

\qquad\qquad\qquad\ \ 

The following is Lemma 43 of \cite{UltrafiltersinRandomModel}.

\begin{lemma}
[\textsf{CH}]Let $M,N\in$\textsf{Sub}$\left(  \kappa\right)  $ with
$\delta_{M}\leq\delta_{N},$ $l\in\omega$ and $A,B\in\left[  \omega_{l}\right]
^{\leq\omega}$ such that $A\in M$ and $B\in N.$ \label{partir conjunto en dos}

\begin{enumerate}
\item There is a partition $P=\left\{  A_{0},A_{1}\right\}  \in M$ of $A$ such
that $A_{0}\in N$ and $A_{1}\cap B=\emptyset.$

\item $A\cap B\in N.$
\end{enumerate}
\end{lemma}

\qquad\ \qquad\ \ 

The following result allows to \textquotedblleft extend\textquotedblright%
\ decisive partitions.

\begin{proposition}
[\textsf{CH}]Assume $l,n\in\omega$ such that $l>0,$ $\left\langle
M_{0},...,M_{n}\right\rangle $ is a $\delta$-increasing sequence and
$P=\left\langle P_{i}\mid i\leq n\right\rangle $ is a decisive partition of
countable subsets of $\omega_{l}$ for $\left\langle M_{0},...,M_{n}%
\right\rangle .$ Let $N$ be an $l$-model such that $\delta_{n}<\delta_{N},$
$D\in N\cap\left[  \omega_{l}\right]  ^{\leq\omega}$ and $X\in N$ be any set.
There are $L\in N\ $and $R=\left\langle R_{i}\mid i\leq n+1\right\rangle $
such that the following conditions hold: \label{Agrandar Particion}

\begin{enumerate}
\item $L\in$ \textsf{Sub}$\left(  \kappa\right)  ,$ $\delta_{n}<\delta_{L}$
and $X,D\in L.$

\item $R$ is a decisive partition for $\left\langle M_{0},...,M_{n}%
,L\right\rangle .$

\item $R_{i}\preceq P_{i}$ for $i\leq n.$

\item $D\subseteq\bigcup R_{n+1}.$\qquad\qquad\qquad\ \ \ \ \ \ \ \ \ \ 
\end{enumerate}
\end{proposition}

\begin{proof}
Apply Corollary \ref{Partir Particion} to find $Q=\left\langle Q_{0}%
,...,Q_{n}\right\rangle $ and $L\in N$ such that:

\begin{enumerate}
\item $L\in$ \textsf{Sub}$\left(  \kappa\right)  ,$ $\delta_{n}<\delta_{L}$
and $X,D\in L.$

\item $Q$ is decisive for $\left\langle M_{0},...,M_{n}\right\rangle .$

\item $Q_{i}\preceq_{2}P_{i}$ for $i\leq n.$

\item Every element of $Q_{i}$ is in $L$ or is disjoint from it.
\end{enumerate}

\qquad\ \ \ 

Let $C\in\bigcup\limits_{i\leq n}Q_{i}$ and $j$ is the minimal such that $C\in
Q_{j}.$ We now invoke Lemma \ref{partir conjunto en dos} to find $C_{0},C_{1}$
with the following properties:

\begin{enumerate}
\item $\left\{  C_{0},C_{1}\right\}  \in M_{j}$ is a partition of $C.$

\item $C_{0}\in L$ and $C_{1}\cap D=\emptyset.$
\end{enumerate}

\qquad\ \qquad\ \ 

For $i\leq n$ define $R_{i}=\left\{  C_{u}\mid C\in Q_{i}\wedge u\in2\right\}
$ and\newline$R_{n+1}=\left(  (\bigcup\limits_{i\leq n}R_{i})\cap L\right)
\cup\{D\setminus\bigcup\limits_{i\leq n}R_{i}\}.$ We claim these items are as
desired. For convenience, denote $M_{n+1}=L.$

\begin{claim}
\qquad\ \ \ \ \ \ \ \ \qquad\ \ \ \ \ \ \ \ \qquad\ \ \ \ \ \ \ 

\begin{enumerate}
\item If $i\leq n,$ then $R_{i}$ is a partition and $R_{i}\in M_{i}$.

\item If $i<j\leq n+1$ and $B\in R_{i},$ then either $B\in R_{j}$ or $B\cap
M_{j}=\emptyset.$

\item $R_{n+1}\in M_{n+1},$ is a partition and $D\subseteq\bigcup R_{n+1}.$
\end{enumerate}
\end{claim}

\qquad\ \ \qquad\ \ 

We prove the first point. Clearly $R_{i}$ is a partition. We now argue that it
is in $M_{i}.$ Let $C\in Q_{i}$, we need to show that $C_{0},C_{1}\in M_{i}.$
Let $j\leq i$ be the first one for which $C\in Q_{j}.$ Since $C\in M_{i}\cap
M_{j}$ and $\delta_{j}\leq\delta_{i},$ it follows by Lemma
\ref{Lema potencia contenida} that $\mathcal{P}\left(  C\right)  \cap
M_{j}\subseteq M_{i},$ so $C_{0},C_{1}\in M_{i}.$

\qquad\qquad\qquad

To prove the second point, first consider the case that $j<n+1.$ Since $B\in
R_{i},$ then $B=C_{u}$ for some $C\in Q_{i}$ and $u\in2.$ If $C\in Q_{j},$
then $C_{u}\in R_{j}$ by definition. In case $C\notin Q_{j},$ we know that
$C\cap M_{j}=\emptyset,$ hence $C_{u}\cap M_{j}=\emptyset.$ In case $j=n+1,$
for $C_{u}\in R_{i},$ we know that either $C\cap L=\emptyset$ or $C\in L.$ In
the latter case, $C_{u}$ is also in $L$ by Lemma \ref{Lema potencia contenida}%
. The third point is trivial.
\end{proof}

\qquad\qquad\ \ \ \ \qquad\ \ 

It is worth noting that for our application, the set $D$ above is unnecessary
(equivalently, we only need the case $D=\emptyset$). However, we wrote the
result in that stronger form for potential future applications. The following
result was proved in \cite{PFiltersCohenRandomLaver} for the case $l=2,$ but
the proof generalizes to any $l\in\omega$ with $l>0$. For completeness, we
include a brief sketch of the argument.

\begin{lemma}
[\textsf{CH}]Assume $l,n\in\omega$ with $l>0,$ $\left\langle M_{0}%
,...,M_{n}\right\rangle $ is a $\delta$-increasing sequence, $P=\left\langle
P_{i}\mid i\leq n\right\rangle $ is a decisive partition for $\left\langle
M_{0},...,M_{n}\right\rangle $ and $\gamma\in M_{n}\cap\omega_{l}.$ There is a
permutation $\triangle:\omega_{l}\longrightarrow\omega_{l}$ with the following
properties: \label{Permutacion}

\begin{enumerate}
\item If $A\in P_{n},$ then $\triangle\upharpoonright A$ is the identity.

\item If $A\in\bigcup\limits_{i<n}P_{i},$ then:

\begin{enumerate}
\item $\triangle\upharpoonright A$ is increasing.

\item $\triangle\left[  A\right]  \in M_{n}.$

\item $\triangle\left[  A\right]  \cap\gamma=\emptyset.$
\end{enumerate}
\end{enumerate}
\end{lemma}

\begin{proof}
By increasing $\gamma$ if necessary, we may assume that $\bigcup
P_{n}\subseteq\gamma.$ Let $R=\bigcup\limits_{i<n}P_{i}\setminus P_{n}$ and
take an enumeration $R=\left\{  A_{j}\mid j<m\right\}  .$ For each $j<m,$
denote $\varepsilon_{j}=$ \textsf{OT}$\left(  A_{j}\right)  $ and $e_{j}%
:A_{j}\longrightarrow\varepsilon_{j}$ the unique isomorphism. Choose $\beta\in
M_{n}\cap\omega_{l}$ an indecomposable ordinal such that $\gamma
,\varepsilon_{0},...,\varepsilon_{m}<\beta.$ Define $\triangle_{0}:$
$\bigcup\limits_{i\leq n}P_{i}\longrightarrow\omega_{l}$ as follows:

\begin{enumerate}
\item If $A\in P_{n},$ then $\triangle_{0}\upharpoonright A$ is the identity.

\item For every $i<m$ and $\alpha\in A_{j},$ we have that $\triangle
_{0}\left(  \alpha\right)  =\beta\left(  j+1\right)  +e_{j}\left(
\alpha\right)  .$
\end{enumerate}

\qquad\qquad\ \ \ \ \qquad\ \ 

It is easy to see that $\triangle_{0}$ is injective, $\triangle_{0}\left[
A_{j}\right]  \in M_{n}$ for every $j<m$ (this is because $\triangle
_{0}\left[  A_{j}\right]  $ is just the interval $[\beta\left(  j+1\right)
,\varepsilon_{j})$). We can now extend $\triangle_{0}$ to a permutation.
\end{proof}

\section{Ultrafilters and Towers \label{Seccion ultrafiltros y torres}}

For this section, $\mathbb{P}\left(  I\right)  $ will denote either
$\mathbb{C}\left(  I\right)  $ or $\mathbb{B}\left(  I\right)  .$ We fix
$l\in\omega$ with $l>1$ and $\kappa$ a large enough regular cardinal. We will
use the technology developed in the previous section to prove that after
adding less than $\aleph_{\omega}$ Cohen or random reals to a model of
\textsf{CH}, there is an ultrafilter that intersects every increasing tower.
We start with a technical definition.

\begin{definition}
We say that $\left(  \overline{M},\overline{W},\overline{P}\right)  $ is a
\emph{nice pattern }if the following conditions hold:

\begin{enumerate}
\item $\overline{M}=\left\langle M_{0},...,M_{n}\right\rangle $ is a $\delta
$-increasing sequence of elements of \textsf{Sub}$\left(  \kappa\right)  .$

\item $\overline{W}=\langle\dot{W}_{0},....,\dot{W}_{n}\rangle$ is a sequence
of $\mathbb{P}\left(  \omega_{l}\right)  $-names of subsets of $\omega$ such
that $\dot{W}_{i}\in M_{i}$ for each $i\leq n.$

\item $P=\left\langle P_{i}\mid i\leq n\right\rangle $ is a decisive partition
for $\left\langle M_{0},...,M_{n}\right\rangle .$

\item $\dot{W}_{i}$ is a $\mathbb{P(}\bigcup P_{i}\mathbb{)}$-name (for $i\leq
n$).

\item $\mathbb{P}\left(  \omega_{l}\right)  $ forces that $\dot{W}_{0}%
\cap....\cap\dot{W}_{n}$ is infinite.
\end{enumerate}
\end{definition}

\qquad\ \qquad\ \ \ 

We will need another result from \cite{UltrafiltersinRandomModel}, which is
Proposition 51 of that paper.

\begin{proposition}
[\textsf{CH}]Let $M,N\in$ \textsf{Sub}$\left(  \kappa\right)  $ with
$\delta_{M}\leq\delta_{N},$ $I\in M\cap\left[  \omega_{l}\right]  ^{\omega}$
and $\dot{a}\in M$ a $\mathbb{P}\left(  I\right)  $-name for a subset of
$\omega.$ If $\triangle:\omega_{l}\longrightarrow\omega_{l}$ a permutation for
which there is $P\in M$ a finite partition of $I$ such that for every $A\in P$
we have that $\triangle\upharpoonright A$ is order preserving and
$\triangle\left[  A\right]  \in N,$ then $\triangle_{\ast}\left(  \dot
{a}\right)  $ is equivalent to a name in $N.$\label{prop mandar nombres}\qquad\ \ \ 
\end{proposition}

\qquad\ \ \ \ \ \ \ 

We now have the following:

\begin{proposition}
[\textsf{CH}]Let $\left(  \overline{M},\overline{W},\overline{P}\right)  $ be
a nice pattern ($\overline{M}=\left\langle M_{0},...,M_{n}\right\rangle ,$
$\overline{W}=\langle\dot{W}_{0},....,\dot{W}_{n}\rangle,$ $P=\left\langle
P_{i}\mid i\leq n\right\rangle $), $\mathcal{A}=\{\dot{A}_{\alpha}\mid
\alpha\in\omega_{1}\}$ be a $\mathbb{P}\left(  \omega_{l}\right)  $-name for
an increasing tower, $N$ an $l$-model such that $\mathcal{A}\in N,$
$\delta_{n}<\delta_{N}$ and $X\in N$ be any set. There are $L\in N,$ $\beta
\in\omega_{1}$ and $R=\left\langle R_{i}\mid i\leq n+1\right\rangle $ with the
following properties: \label{extender nice pattern}

\begin{enumerate}
\item $L\in$ \textsf{Sub}$\left(  \kappa\right)  ,$ $\delta_{n}<\delta_{L}$
and $\beta,\mathcal{A},X\in L.$

\item $R_{i}\preceq P_{i}$ for every $i\leq n.$

\item $(M_{0},...,M_{n},L,\dot{W}_{0},....,\dot{W}_{n},\dot{A}_{\beta},R)$ is
a nice pattern.
\end{enumerate}
\end{proposition}

\begin{proof}
By Proposition \ref{Agrandar Particion} (with $D=\emptyset$) we know there are
$L\in N$ a submodel and $R=\left\langle R_{i}\mid i\leq n+1\right\rangle $
with the following properties:

\begin{enumerate}
\item $\delta_{n}<\delta_{L}$ and $X,\mathcal{A}\in L.$

\item $R$ is a decisive partition for $\left\langle M_{0},...,M_{n}%
,L\right\rangle .$

\item $R_{i}\preceq P_{i}$ for every $i\leq n.$
\end{enumerate}

\qquad\qquad\ \ 

We can assume that $R_{n+1}$ consists exactly of the elements of
$\bigcup\limits_{i\leq n}R_{i}$ that are in $L.$ Since $\mathcal{A}\in L$ and
each $\dot{A}_{\alpha}$ has a countable support, we can find $\gamma\in
L\cap\omega_{l}$ such that for every $\alpha\in\omega_{1},$ the name $\dot
{A}_{\alpha}$ has a support contained in $\gamma$ (recall that $l>1$). We can
now apply Lemma \ref{Permutacion} and find a $\triangle:\omega_{l}%
\longrightarrow\omega_{l}$ with the following properties:

\begin{enumerate}
\item If $B\in R_{n+1},$ then $\triangle\upharpoonright B$ is the identity mapping.

\item If $B\in\bigcup\limits_{i\leq n}R_{i},$ then the following holds:

\begin{enumerate}
\item $\triangle\upharpoonright B$ is increasing.

\item $\triangle\left[  B\right]  \in L.$

\item $\triangle\left[  B\right]  \cap\gamma=\emptyset.$
\end{enumerate}
\end{enumerate}

\qquad\ \ 

By Proposition \ref{prop mandar nombres}, we conclude that $\triangle_{\ast
}(\dot{W}_{0}),...,\triangle_{\ast}(\dot{W}_{n})\in L.$ Since $\mathbb{P}%
\left(  \omega_{l}\right)  $ forces that $\dot{W}_{0}\cap....\cap\dot{W}_{n}$
is infinite and $\triangle_{\ast}$ is an automorphism, we get that
$\triangle_{\ast}(\dot{W}_{0})\cap...\cap\triangle_{\ast}(\dot{W}_{n})$ is
also forced to be infinite (see Proposition \ref{Prop formula isomorfismo}).
Now, since $\mathcal{A}\in L,$ $\mathbb{P}\left(  \omega_{l}\right)  $ is ccc
and $\mathcal{A}$ is forced to be an increasing tower, \ we know that there is
$\beta\in L\cap\omega_{1}$ for which $\mathbb{P}\left(  \omega_{l}\right)  $
forces that $\bigcap\limits_{i\leq n}\triangle_{\ast}(\dot{W}_{i})\cap\dot
{A}_{\beta}$ is infinite. Let \textsf{sup}$(\dot{A}_{\beta})\in L$ be a
countable support of $\dot{A}_{\beta}$ and $I=$ \textsf{sup}$(\dot{A}_{\beta
})\setminus\bigcup R_{n+1}$. Define $\sigma_{0}:I\cup\bigcup\bigcup
\limits_{i\leq n+1}R_{i}\longrightarrow\omega_{l}$ with the following properties:

\begin{enumerate}
\item $\sigma_{0}\upharpoonright I$ is the identity

\item $\sigma_{0}\upharpoonright J=\triangle\upharpoonright J$ if $J\in
\bigcup\limits_{i\leq n+1}R_{i}.$
\end{enumerate}

\qquad\ \ 

We claim that $\sigma_{0}$ is injective. Since $\triangle$ is injective, it
remains to check that if $\alpha\in I$ and $\xi\in J$ (for some $J\in
\bigcup\limits_{i\leq n+1}R_{i}$), then $\sigma_{0}\left(  \alpha\right)
\neq\sigma_{0}\left(  \xi\right)  .$ We proceed by cases. If $J\notin L,$ then
$\sigma_{0}\left(  \alpha\right)  =\alpha<\gamma,$ while $\sigma_{0}\left(
\xi\right)  \geq\gamma.$ While if $J\in L,$ we know that $\sigma_{0}\left(
\alpha\right)  =\alpha$ and $\sigma_{0}\left(  \xi\right)  =\xi,$ so we are
done. Find $\sigma:\omega_{l}\longrightarrow\omega_{l}$ a permutation
extending $\sigma_{0}.$

\begin{claim}
\qquad\qquad\qquad\qquad\qquad\ \ \ \ \ \ \ \qquad\ \ \ \ \ \ \ \ \ \ \ \ \ \ \ 

\begin{enumerate}
\item $\sigma\upharpoonright$ \textsf{sup}$(\dot{A}_{\beta})$ is the identity.

\item $\sigma_{\ast}(\dot{A}_{\beta})=\dot{A}_{\beta}.$

\item $\sigma_{\ast}(\dot{W}_{i})=\triangle_{\ast}(\dot{W}_{i})$ for $i\leq
n.$
\end{enumerate}
\end{claim}

\qquad\ \ 

The second and third points are consequence of the first one and Lemma
\ref{automorfismo coinciden}. We prove the first point. Let $\alpha\in$
\textsf{sup}$(\dot{A}_{\beta}).$ In case $\alpha\in I,$ we know that
$\sigma\left(  \alpha\right)  =\sigma_{0}\left(  \alpha\right)  =\alpha.$ If
$\alpha\notin I,$ then there is $J\in$ $\bigcup R_{n+1}$ with $\alpha\in J.$
Since $J\in L,$ we conclude that $\sigma\left(  \alpha\right)  =\triangle
\left(  \alpha\right)  =\alpha.$ This finishes the proof of the claim.

\qquad\qquad\qquad

We already knew that $\mathbb{P}\left(  \omega_{l}\right)  $ forces that
$\bigcap\limits_{i\leq n}\triangle_{\ast}(\dot{W}_{i})\cap\dot{A}_{\beta}=$
$\bigcap\limits_{i\leq n}\sigma_{\ast}(\dot{W}_{i})\cap\sigma_{\ast}(\dot
{A}_{\beta})$ is infinite. Since $\sigma_{\ast}$ is an isomorphism (see
Proposition \ref{Prop formula isomorfismo}), it follows that $\bigcap
\limits_{i\leq n}\dot{W}_{i}\cap\dot{A}_{\beta}$ is forced to be infinite.
Finally, define $\overline{R}_{n+1}=R_{n+1}\cup\left\{  I\right\}  .$ It
follows that $L,$ $\langle R_{0},...,R_{n},\overline{R}_{n+1}\rangle$ and
$\beta$ are as desired.
\end{proof}

\qquad\ \ \qquad\ \ 

We can now prove our main result:

\begin{theorem}
[\textsf{CH}]Let $l\in\omega.$ $\mathbb{P}\left(  \omega_{l}\right)  $ forces
that there is an ultrafilter that intersects all increasing towers.
\label{Teorema ultrafiltro torre}
\end{theorem}

\begin{proof}
If $l\leq1,$ the forcing extension will be a model of \textsf{CH, }so the
result follows by Theorem \ref{KunenvanMill}, so we now assume $l>1.$ Let
$\mathbb{T}$ be the the collection of all $\mathbb{P}\left(  \omega
_{l}\right)  $-names for increasing towers. For $\mathcal{A}\in\mathbb{T},$ we
enumerate it as $\mathcal{A=\{}\dot{A}_{\alpha}(\mathcal{A})\mid\alpha
\in\omega_{1}\mathcal{\}}$ \ (recall Theorem \ref{KunenTorres}) and fix the
following objects:

\begin{enumerate}
\item $N(\mathcal{A})$ an $l$-model such that $\mathcal{A\in}$ $N\left(
\mathcal{A}\right)  .$

\item $\delta\left(  \mathcal{A}\right)  =\delta_{N\left(  \mathcal{A}\right)
}.$

\item $\dot{A}_{\mathcal{A}}=\dot{A}\mathcal{_{\delta\left(  \mathcal{A}%
\right)  }(\mathcal{A})}.$
\end{enumerate}

\qquad\ \qquad\ \ 

Let $\mathcal{B=\{}\dot{A}\mathcal{_{\mathcal{A}}\mid A}\in\mathbb{T}%
\mathcal{\}}.$ We claim that $\mathcal{B}$ is forced to be a centered family.
Once this is proved, any ultrafilter extending $\mathcal{B}$ will be as desired.

\begin{claim}
Let $\mathcal{A}_{0},...,\mathcal{A}_{n}\in\mathbb{T}$ such that
$\delta(\mathcal{A}_{0})\leq...\leq\delta(\mathcal{A}_{n}).$ There
are\newline$\left\langle M_{0},...,M_{n}\right\rangle ,$ $P=\left\langle
P_{i}\mid i\leq n\right\rangle $ and $\beta_{0},...,\beta_{n}$ with the
following properties:

\begin{enumerate}
\item $M_{i}\in N(\mathcal{A}_{i})$ and $\beta_{i}<\delta_{M_{i}}.$

\item $(M_{0},...,M_{n},\dot{A}_{\beta_{0}}(\mathcal{A}_{0}),...,\dot
{A}_{\beta_{n}}(\mathcal{A}_{n}),P)$ is a nice pattern.
\end{enumerate}
\end{claim}

\qquad\ \qquad\ 

Proceed by induction on $n.$ In case $n=0,$ take any $M_{0}\in N(\mathcal{A}%
_{0})$ and any $\beta_{0}<\delta_{M_{0}}.$ Assume the claim is true for $n,$
we will prove it is true for $n+1$ as well. We have $\mathcal{A}%
_{0},...,\mathcal{A}_{n+1}\in\mathbb{T}$ such that $\delta(\mathcal{A}%
_{0})\leq...\leq\delta(\mathcal{A}_{n+1}).$ By the inductive hypothesis, find
$\left\langle M_{0},...,M_{n}\right\rangle ,$ $P=\left\langle P_{i}\mid i\leq
n\right\rangle $ and $\beta_{0},...,\beta_{n}$ as in the claim (for
$\mathcal{A}_{0},...,\mathcal{A}_{n}$). Note that $\delta_{M_{n}}%
<\delta(\mathcal{A}_{n})\leq\delta(\mathcal{A}_{n+1}).$ We can apply
Proposition \ref{extender nice pattern}. This finishes the proof of the claim.

\qquad\qquad\qquad

We can now prove that $\mathcal{B}$ is forced to be centered. Let
$\mathcal{A}_{0},...,\mathcal{A}_{n}\in\mathbb{T}.$ Apply the claim and find
the objects as above. Since the sequence\newline$(M_{0},...,M_{n},\dot
{A}_{\beta_{0}}(\mathcal{A}_{0}),...,\dot{A}_{\beta_{n}}(\mathcal{A}_{n}),P)$
is a nice pattern, so it follows that the set $\bigcap\limits_{i\leq n}\dot
{A}_{\beta_{i}}(\mathcal{A}_{i})$ is forced to be infinite. Finally, since
$\dot{A}_{\beta_{i}}(\mathcal{A}_{i})$ is forced to be an almost subset of
$\dot{A}_{\mathcal{A}_{i}},$ we conclude that $\bigcap\limits_{i\leq n}\dot
{A}_{\mathcal{A}_{i}}$ is forced to be infinite.
\end{proof}

\section{No covering with nowhere dense P-sets \label{Seccion Covering}}

It is well-known that Cohen forcing preserves towers. Moreover, in a Cohen
extension (over a model of \textsf{CH}) for every tall \textsf{P}-ideal
$\mathcal{I}$ in the extension, we can find an intermediate model where the
Continuum Hypothesis holds and $\mathcal{I}$ remains a tall \textsf{P}-ideal.
Consequently, by Proposition \ref{CH pideal contiene torre} it contains an
increasing tower, which remains a tower in the final extension (see
\cite{PFiltersCohenRandomLaver} and \cite{SetTheoryinTopology} for more
details). This leads to the following result:

\begin{proposition}
[\textsf{CH}]In any Cohen extension, every tall \textsf{P}-ideal contains an
increasing tower. \label{P ideales modelo de Cohen}
\end{proposition}

\qquad\ \qquad\ \ \ \ \ \ 

When combined with Theorem \ref{Teorema ultrafiltro torre}, we get the following:

\begin{theorem}
[\textsf{CH}]Let $l\in\omega.$ $\mathbb{C}\left(  \omega_{l}\right)  $ forces
that $\omega^{\ast}$ can not be covered by nowhere dense \textsf{P}-sets.
\label{Teorema para Cohen}
\end{theorem}

\qquad\qquad\qquad\qquad\ \ \ 

What about random forcing? While it is true that all towers in the random
models have length $\omega_{1}$, Proposition \ref{P ideales modelo de Cohen}
fails for random forcing. The issue is that random forcing may fail to
preserve towers. Specifically, random forcing diagonalizes the \emph{summable
ideal} $\mathcal{J}_{\frac{1}{n}}=\{A\subseteq\omega\mid\sum\limits_{n\in
A}\frac{1}{n+1}<\infty\}$ (see \cite{CardinalInvariantsofAnalyticPIdeals}).
Since $\mathcal{J}_{\frac{1}{n}}$ is a \textsf{P}-ideal\textsf{, }it contains
an increasing tower (assuming \textsf{CH}), which will be destroyed after
adding a single random real. Consequently, if we add at least $\omega_{2}$
random reals to a model of \textsf{CH, }the summable ideal will be a
\textsf{P}-ideal that does not contain an increasing tower.

\section{Open Questions}

We do not know if Theorem \ref{Teorema para Cohen} holds for random forcing:

\begin{problem}
Assume \textsf{CH }and let $l\in\omega.$ Does $\mathbb{B}\left(  \omega
_{l}\right)  $ force that $\omega^{\ast}$ can not be covered by nowhere dense
\textsf{P}-sets?
\end{problem}

\qquad\ \qquad\ \ 

We do not know if our results can be extended when we add more than
$\aleph_{\omega}$ Cohen reals.

\begin{problem}
Assume \textsf{CH }and let $\kappa$ be any uncountable cardinal. Does
$\mathbb{C}\left(  \kappa\right)  $ force that $\omega^{\ast}$ can not be
covered by nowhere dense \textsf{P}-sets?
\end{problem}

\qquad\ \qquad\ \ 

Finally, we would like to reiterate the following problems from
\cite{CozeroAccesiblePoints}, since we consider them very important:

\begin{problem}
Does \textsf{PFA }imply that $\omega^{\ast}$ can not be covered by nowhere
dense \textsf{P}-sets?
\end{problem}

\begin{problem}
Is it consistent with \textsf{MA }and the negation of \textsf{CH} that
$\omega^{\ast}$ can be covered by nowhere dense \textsf{P}-sets?
\end{problem}

\qquad\ \ \qquad\ \ \ \qquad\ \qquad\ \ \ \ \ \qquad\ \ \ \ 

\begin{acknowledgement}
We would like to thank Michael Hru\v{s}\'{a}k for several helpful discussions
related to the topic of this paper. We would also like to thank the referee
for their comments, which improved the paper.
\end{acknowledgement}

\bibliographystyle{plain}

\begin{thebibliography}{10}

\bibitem{MoreonNowhereDensePsets}
Bohuslav Balcar, Ryszard Frankiewicz, and Charles Mills.
\newblock More on nowhere dense closed {$P$}-sets.
\newblock {\em Bull. Acad. Polon. Sci. S\'er. Sci. Math.}, 28(5-6):295--299,
  1980.

\bibitem{MeasuresinducebbyForcingnames}
Piotr Borodulin-Nadzieja and Katarzyna Cegie\l~ka.
\newblock On measures induced by forcing names for ultrafilters.
\newblock {\em Topology Appl.}, 323:Paper No. 108278, 14, 2023.

\bibitem{TowersinFilters}
J\"org Brendle, Barnab\'as Farkas, and Jonathan Verner.
\newblock Towers in filters, cardinal invariants, and {L}uzin type families.
\newblock {\em J. Symb. Log.}, 83(3):1013--1062, 2018.

\bibitem{AlanSubmodelos}
Alan Dow.
\newblock An introduction to applications of elementary submodels to topology.
\newblock {\em Topology Proc.}, 13(1):17--72, 1988.

\bibitem{SetTheoryinTopology}
Alan Dow.
\newblock Set theory in topology.
\newblock In {\em Recent progress in general topology ({P}rague, 1991)}, pages
  167--197. North-Holland, Amsterdam, 1992.

\bibitem{CozeroAccesiblePoints}
Alan Dow.
\newblock Cozero-accessible points.
\newblock {\em Topology Appl.}, 156(16):2609--2613, 2009.

\bibitem{PFiltersCohenRandomLaver}
Alan Dow.
\newblock P-filters and {C}ohen, random, and {L}aver forcing.
\newblock {\em Topology Appl.}, 281:107200, 16, 2020.

\bibitem{UltrafiltersinRandomModel}
Alan Dow and Osvaldo Guzmán.
\newblock Ultrafilters in the random real model.
\newblock {\em arXiv preprint arXiv:2507.08346}, 2025.

\bibitem{NowhereDenseCCCPSets}
Alan Dow and Jan van Mill.
\newblock On nowhere dense ccc {$P$}-sets.
\newblock {\em Proc. Amer. Math. Soc.}, 80(4):697--700, 1980.

\bibitem{ConcerningRingsofFunctions}
Leonard Gillman and Melvin Henriksen.
\newblock Concerning rings of continuous functions.
\newblock {\em Trans. Amer. Math. Soc.}, 77:340--362, 1954.

\bibitem{LibroBooleanAlgebras}
Steven Givant and Paul Halmos.
\newblock {\em Introduction to {B}oolean algebras}.
\newblock Undergraduate Texts in Mathematics. Springer, New York, 2009.

\bibitem{CardinalInvariantsofAnalyticPIdeals}
Fernando Hern\'andez-Hern\'andez and Michael Hru{\v{s}}{\'a}k.
\newblock Cardinal invariants of analytic {$P$}-ideals.
\newblock {\em Canad. J. Math.}, 59(3):575--595, 2007.

\bibitem{Jech}
Thomas Jech.
\newblock {\em Set theory}.
\newblock Springer Monographs in Mathematics. Springer-Verlag, Berlin, 2003.
\newblock The third millennium edition, revised and expanded.

\bibitem{JechAC}
Thomas~J. Jech.
\newblock {\em The axiom of choice}, volume Vol. 75 of {\em Studies in Logic
  and the Foundations of Mathematics}.
\newblock North-Holland Publishing Co., Amsterdam-London; American Elsevier
  Publishing Co., Inc., New York, 1973.

\bibitem{KomjathBook}
P\'{e}ter Komj\'{a}th and Vilmos Totik.
\newblock {\em Problems and theorems in classical set theory}.
\newblock Problem Books in Mathematics. Springer, New York, 2006.

\bibitem{oldKunen}
Kenneth Kunen.
\newblock {\em Set theory}, volume 102 of {\em Studies in Logic and the
  Foundations of Mathematics}.
\newblock North-Holland Publishing Co., Amsterdam-New York, 1980.
\newblock An introduction to independence proofs.

\bibitem{KunenCohenyRandom}
Kenneth Kunen.
\newblock Random and {C}ohen reals.
\newblock In {\em Handbook of set-theoretic topology}, pages 887--911.
  North-Holland, Amsterdam, 1984.

\bibitem{Kunen}
Kenneth Kunen.
\newblock {\em Set theory}, volume~34 of {\em Studies in Logic (London)}.
\newblock College Publications, London, 2011.

\bibitem{NowhereDenseClosedPsets}
Kenneth Kunen, Jan van Mill, and Charles~F. Mills.
\newblock On nowhere dense closed {$P$}-sets.
\newblock {\em Proc. Amer. Math. Soc.}, 78(1):119--123, 1980.

\bibitem{betaomega}
Jan van Mill.
\newblock An introduction to {$\beta\omega$}.
\newblock In {\em Handbook of set-theoretic topology}, pages 503--567.
  North-Holland, Amsterdam, 1984.

\bibitem{RemarkonNowheredensePsets}
Jian-Ping Zhu.
\newblock A remark on nowhere dense closed {$P$}-sets.
\newblock Number 823, pages 91--100. 1993.
\newblock General topology, geometric topology and related problems (Japanese)
  (Kyoto, 1992).

\end{thebibliography}
\def\cprime{$'$}

\qquad\qquad\qquad\ \ \ \ \ \ \ \ \ \ \ \ \ \ \ \ \ \ \ \qquad\qquad\qquad\ \ \ 

Alan Dow

Department of Mathematics and Statistics, UNC Charlotte.

adow@charlotte.edu

\qquad\qquad\qquad\ \ \ \ \ \ \ \ \ \ \ \ \qquad\ \ \ \ \ \qquad
\qquad\ \ \ \qquad\qquad\qquad\qquad\ \ \ \ \ \ 

Osvaldo Guzm\'{a}n

Centro de Ciencias Matem\'{a}ticas, UNAM.

oguzman@matmor.unam.mx

\end{document}